\documentclass[final]{gSSR2e}

\usepackage{natbib}
\usepackage{amssymb}
\usepackage{color}
\usepackage[%
  bookmarks=true,%
  bookmarksopen=true,%
  breaklinks=true,%
  colorlinks=true,%
  linkcolor=blue,anchorcolor=blue,%
  citecolor=blue,filecolor=blue,%
  menucolor=blue,pagecolor=blue,%
  urlcolor=blue]{hyperref}

\begin{document}


 \renewcommand{\theequation}{\thesection.\arabic{equation}}
 \newcommand{\eps}{\varepsilon}
 \newcommand{\lam}{\lambda}
 \newcommand{\To}{\rightarrow}
 \newcommand{\as}{{\rm d}P\times {\rm d}t-a.s.}
 \newcommand{\ps}{{\rm d}P-a.s.}
 \newcommand{\jf}{\int_t^T}
 \newcommand{\tim}{\times}

 \newcommand{\F}{\mathcal{F}}
 \newcommand{\E}{\mathbf{E}}
 \newcommand{\N}{\mathbf{N}}
 \newcommand{\s}{\mathcal{S}}
 \newcommand{\T}{[0,T]}
 \newcommand{\LT}{L^2(\Omega,\F_T, P)}
 \newcommand{\Lt}{L^2(\Omega,\F_t, P)}
 \newcommand{\Ls}{L^2(\Omega,\F_s, P)}
 \newcommand{\Lr}{L^2(\Omega,\F_r, P)}
 \newcommand{\Lp}{L^p(\R^k)}
 \newcommand{\Lps}{L^2(\R^k)}
 \newcommand{\R}{{\bf R}}
 \newcommand{\RE}{\forall}

\newcommand{\gET}[1]{\underline{{\mathcal {E}}_{t,T}^g}[#1]}
\newcommand{\gEt}[1]{\underline{{\mathcal {E}}_{s,t}^g}[#1]}
\newcommand{\gETo}[1]{\underline{{\mathcal {E}}_{0,T}^g}[#1]}

\newcommand{\sgET}[1]{\overline{{\mathcal {E}}_{t,T}^g}[#1]}
\newcommand{\sgEt}[1]{\overline{{\mathcal {E}}_{s,t}^g}[#1]}
\newcommand{\sgETo}[1]{\overline{{\mathcal {E}}_{0,T}^g}[#1]}

\newcommand {\Lim}{\lim\limits_{n\rightarrow \infty}}
\newcommand {\Limk}{\lim\limits_{k\rightarrow \infty}}
\newcommand {\Limm}{\lim\limits_{m\rightarrow \infty}}
\newcommand {\llim}{\lim\limits_{\overline{n\rightarrow \infty}}}
\newcommand {\slim}{\overline{\lim\limits_{n\rightarrow \infty}}}
\newcommand {\Dis}{\displaystyle}

\newtheorem{Remark}[theorem]{\scshape\bf Remark}
\def\Remarkfont{\hbox{\hspace{-4pt}$\bm{:}\,\,\,\,\,$}}

\title {$L^p$ solutions of finite and infinite time interval BSDEs with non-Lipschitz coefficients}

\author{ShengJun Fan$^{\ast}$
\thanks{$^\ast$Corresponding author. Email: f\_s\_j@126.com\vspace{6pt}}
and Long Jiang\\\vspace{6pt}
{\em{College of Sciences, China University of Mining $\&$ Technology, Xuzhou, Jiangsu 221116,
P.R. China}}\\\vspace{6pt}
\received{9 February 2010}}

\maketitle


\begin{abstract}
In this paper, we are interested in solving multidimensional backward stochastic differential equations (BSDEs) in $L^p\ (p>1)$ under weaker assumptions on the coefficients, considering
both a finite and an infinite time interval. We establish a general existence and uniqueness result of solutions in $L^p\ (p>1)$ to finite and infinite time interval BSDEs with non-Lipschitz
coefficients, which includes the corresponding results in
\citet{Par90}, \citet{Mao95}, \citet{Chen97}, \citet{Cons01}, \citet{Wang03}, \citet{Chen00} and \citet{Wang09} as its particular
cases.
\end{abstract}


\begin{classcode}
primary 60H10
\end{classcode}\bigskip


\section{Introduction}\label{introduction}
In this paper, we consider the following multidimensional backward stochastic differential
equation (BSDE for short in the remaining):
\begin{equation}\label{BSDE1}
y_t=\xi+\int_t^T g(s,y_s,z_s)\ {\rm d}s-\int_t^T z_s\ {\rm d}B_s,\ \ t\in\T,
\end{equation}
where the time horizon $T>0$ is a finite or infinite constant, the terminal condition $\xi$ is a
$k$-dimensional random variable, the generator $g(\omega,t,y,z):\Omega\tim\T\tim {\R}^{k}\tim{\R}^{k\tim d}\To{\R}^k$ is progressively
measurable for each $(y,z)$, and $B$ is a
$d$-dimensional Brownian motion. The solution ($y_{\cdot},z_{\cdot}$) is a pair of adapted
processes. The $(\xi,T,g)$ describe the coefficients (parameters) of BSDE \eqref{BSDE1}.\vspace{0.2cm}

Such equations, in the nonlinear case, were first introduced by \citet{Par90}, who established an existence and uniqueness result for solutions in $L^2$ to BSDEs under the Lipschitz assumption of the generator $g$. Since then, BSDEs have attracted much interest, and have become an important mathematical tool in many fields including financial mathematics, stochastic games and optimal control. In particular, much effort has been
made to relax the Lipschitz hypothesis on $g$, for instance, in the one dimensional setting ($k$=1), \citet{Lep97} have proved the existence of a solution in $L^2$ for BSDE \eqref{BSDE1} when
$g$ is continuous and of linear growth in $(y,z)$, \citet{Kob00} obtained the existence and uniqueness of a solution in $L^2$ when $g$ has a quadratic growth in $z$ and the terminal condition $\xi$ is bounded, and then \citet{Bri06} and \citet{Bri08} further extended the result of \citet{Kob00} to the case of unbounded terminal conditions.\vspace{0.2cm}

\citet{Mao95} proposed the following non-Lipschitz assumption for the generator $g$ of multidimensional BSDEs:
\begin{enumerate}
  \item[(H1)] $\as,\ \RE y_1,y_2\in \R^k,z_1,z_2\in\R^{k\times d},$
  $$|g(\omega,t,y_1,z_1)-g(\omega,t,y_2,z_2)|^2\leq\kappa(|y_1-y_2|^2)+c|z_1-z_2|^2,$$
\end{enumerate}
where $c>0$ and $\kappa(\cdot)$ is a concave and nondecreasing function from $\R^+$ to $\R^+$ such that $\kappa(0)=0$, $\kappa(u)>0$ for $u>0$ and $\int_{0^+}\kappa^{-1}(u)\ {\rm d}u=+\infty$. Under this assumption, he proved that BSDE \eqref{BSDE1} with $0<T<+\infty$ has a unique solution in $L^2$.\vspace{0.2cm}

\citet{Wang03} proposed another non-Lipschitz condition for the generator $g$ of multidimensional BSDEs and \citet{Wang09} further generalized it as follows:
\begin{enumerate}
  \item[(H2)] $\as,\ \RE y_1,y_2\in \R^k,z_1,z_2\in\R^{k\times d},$
  $$|g(\omega,t,y_1,z_1)-g(\omega,t,y_2,z_2)|^2\leq\kappa(t,|y_1-y_2|^2)+c|z_1-z_2|^2,$$
\end{enumerate}
where $c>1$ and $\kappa(\cdot,\cdot)\in {\bf S}[T,a(\cdot),b(\cdot)]$, here and henceforth for
$0<T\leq +\infty$, ${\bf S}[T,a(\cdot),b(\cdot)]$ denotes the set of functions $\kappa(\cdot,\cdot):\T\times \R^+\mapsto \R^+$ satisfying the following two conditions:
\begin{itemize}
  \item For fixed $t\in\T$, $\kappa(t,\cdot)$ is a continuous, concave and nondecreasing
        function with $\kappa(t,0)=0$, and for each $t\in\T$, $\kappa(t,u)\leq a(t)+b(t)u$,
        where the functions $a(\cdot),b(\cdot):\T\mapsto \R^+$ satisfy
        $\int_0^T\ [a(t)+b(t)]\ {\rm d}t<+\infty$;\vspace{0.2cm}
  \item The following ODE, $u'(t)=-\kappa(t,u),t\in\T$ with $u(T)=0$, has a unique solution
        $u(t)=0,\ t\in\T$.
\end{itemize}

Under (H2), they proved the existence and uniqueness of the solution in $L^2$ to BSDE \eqref{BSDE1} with $0<T<+\infty$ and, with the help of Bihari's inequality, they proved that their result includes that of \citet{Mao95}.\vspace{0.2cm}

Moreover, \citet{Chen97} and \citet{Chen00} proposed the following non-uniformly Lipschitz condition for the generator $g$ of multidimensional BSDEs:
\begin{enumerate}
  \item[(H3)] $\as,\ \RE y_1,y_2\in \R^k,z_1,z_2\in\R^{k\times d},$
  $$|g(\omega,t,y_1,z_1)-g(\omega,t,y_2,z_2)|\leq u(t)|y_1-y_2|+v(t)|z_1-z_2|,$$
\end{enumerate}
where $u(\cdot),v(\cdot):\T\mapsto \R^+$ satisfy $\int_0^T u(t)\ {\rm d}t<+\infty$ and $\int_0^T v^2(t)\ {\rm d}t<+\infty$.\vspace{0.2cm}

Under (H3), they established the existence and uniqueness of the solution in $L^2$ to BSDE \eqref{BSDE1} with $0< T\leq +\infty$.\vspace{0.2cm}

Furthermore, in the case where $0<T<+\infty$, \citet{Par99} established the existence and uniqueness result of a solution in $L^2$ for BSDE \eqref{BSDE1} where $g$ satisfies the particular monotonicity condition in $y$. Using the same monotonicity condition for $g$, \citet{Bri03} investigated
the existence and uniqueness of a solution in $L^p\ (p>1)$ for BSDE \eqref{BSDE1}.\vspace{0.2cm}

In this paper, we are interested in solving BSDEs in $L^p\ (p>1)$ under weaker assumptions on the coefficients, considering both a finite and an infinite time interval. We establish a
general existence and uniqueness theorem of solutions in $L^p\ (p>1)$ to finite time and infinite interval BSDEs (see Theorem \ref{theoremmainresult} in Section \ref{sectionmainresult}), which includes the corresponding results in \citet{Par90}, \citet{Mao95}, \citet{Chen97}, \citet{Cons01}, \citet{Wang03}, \citet{Chen00} and \citet{Wang09} as its particular cases. The paper is organized as follows. We introduce some preliminaries and lemmas in Section \ref{sectionpreliminariesandlemmas} and put forward and prove our main results in Section \ref{sectionmainresult}. Some examples, corollaries and remarks are given in Section \ref{sectionexamplecorollaryremark} to show that Theorem \ref{theoremmainresult} of this paper is a generalization of some results mentioned above. Finally, some further discussions with respect to our main result are provided in Section \ref{sectionexamplecorollaryremark}.\vspace{-0.2cm}


\section{Preliminaries and Lemmas}\label{sectionpreliminariesandlemmas}\setcounter{equation}{0}
Let us first introduce some notation. First of all, let us fix two real numbers $0<T\leq +\infty$ and $p>1$, and two positive integers $k$ and $d$. Let $(\Omega,\F,P)$ be a probability space carrying a standard $d$-dimensional Brownian motion $(B_t)_{t\geq 0}$. Let
$(\F_t)_{t\geq 0}$ be the natural $\sigma$-algebra generated by $(B_t)_{t\geq 0}$ and $\F=\F_T$. In this paper, the Euclidean norm of a vector $y\in\R^k$ will be defined by $|y|$, and for a $k\times d$ matrix $z$, we define $|z|=\sqrt{{\rm Tr}(zz^*)}$, where $z^*$ is the transpose of $z$. Let $a\wedge b$ represent the minimum of $a,b\in \R$ and  $\langle x,y\rangle$ the inner product of $x,y\in\R^k$. We denote by $\Lp$ the set of all $\R^k$-valued and $\F_T$-measurable random vectors $\xi$ such that $\E[|\xi|^p]<+\infty$. Let ${\s}^p(0,T;\R^k)$ denote the set of $\R^k$-valued, adapted and continuous processes
$(Y_t)_{t\in\T}$ such that
$$\|Y\|_{{\s}^p}:=\left\{\E\left[\sup_{t\in\T} |Y_t|^p\right] \right\}^{1/p}<+\infty.$$
Moreover, let ${\rm M}^p(0,T;\R^{k})$ (resp. ${\rm M}^p(0,T;\R^{k\times d})$) denote the set
of $(\F_t)$-progressively measurable ${\R}^{k}$-valued (${\R}^{k\times d}$-valued) processes $(Z_t)_{ t\in\T}$ such that
$$\|Z\|_{{\rm M}^p}:=\left\{ \E\left[\left(\int_0^T |Z_t|^2\
{\rm d}t\right)^{p/2}\right] \right\}^{1/p}<+\infty.$$
Obviously, both ${\s}^p$ and ${\rm M}^p$ are Banach spaces. As mentioned in the introduction, we will deal only with BSDEs which are equations of type \eqref{BSDE1}, where the terminal
condition $\xi$ belongs to the space $L^p(\R^k)$, and the generator $g$ is $(\F_t)$-progressively measurable for each $(y,z)$.

\begin{definition}\label{definitionBSDEsolution}
  A pair of processes $(y_t,z_t)_{t\in\T}$ is called a solution in $L^p$ to BSDE \eqref{BSDE1}, if $(y_t,z_t)_{t\in\T}\in {\s}^p(0,T;\R^{k})\times {\rm M}^p(0,T;\R^{k\times d})$
  and satisfies \eqref{BSDE1}.
\end{definition}

Let us introduce the following ``Backward Gronwall Inequality". We omit the standard proof.

\begin{lemma}\label{lemma1}
  Let $0<T\leq \infty$, $\alpha(t):\T\mapsto \R^+$ be a decreasing function, $\beta(t):\T\mapsto \R^+$ satisfy $\int_0^T \beta(s)\ {\rm d}s<+\infty$, and $u(t):\T\mapsto \R^+$ be a continuous function with $\sup_{t\in [0,T]}u(t)<+\infty$ such that
  $$u(t)\leq \alpha(t)+\int_t^T \beta(s)u(s)\ {\rm d}s,\ \ t\in [0,T].$$
  Then we have
  $$u(t)\leq \alpha(t)e^{\int_t^T \beta(s)\ {\rm d}s},\ t\in [0,T].$$
\end{lemma}

The following Lemma \ref{lemma2} comes from Corollary 2.3 of \citet{Bri03}, which is the
starting point of this paper.

\begin{lemma}\label{lemma2}
  If $(y_t,z_t)_{t\in\T}$ be a solution in $L^p$ of BSDE \eqref{BSDE1}, $c(p)=p/2[(p-1)\wedge 1]$ and $0\leq t\leq T$, then
  $$\begin{array}{lll}
      \Dis |y_t|^p+c(p)\int_t^T |y_s|^{p-2}1_{|y_s|\neq 0}|z_s|^2\ {\rm d}s &\leq & \Dis|\xi|^p +p\int_t^T |y_s|^{p-2}1_{|y_s|\neq 0}\langle y_s,g(s,y_s,z_s)\rangle\ {\rm d}s\\
          & & \Dis \ -p\int_t^T |y_s|^{p-2}1_{|y_s|\neq 0}\langle y_s,z_s{\rm d}B_s\rangle.
  \end{array}$$
\end{lemma}

Now we establish the following two propositions important in the proof of our main result. In stating these propositions it will be useful to introduce the following assumption on the generator $g$:
\begin{enumerate}
  \item[\textbf{(A)}] $\as,\RE\ (y,z)\in\R^k\tim\R^{k\tim d}$,
       $$|g(\omega,t,y,z)|\leq\mu(t)\left[\psi^{1\over p}(t,|y|^p)+\varphi_t\right]+\nu(t)
       |z|+f_t,$$
\end{enumerate}
where
$\mu(\cdot),\nu(\cdot):\T\mapsto \R^+$ with $\int_0^T[\mu^{p\over p-1}(t)+\nu^2(t)]\,{\rm d}t<+\infty$, both $\varphi_t$
and $f_t$ are nonnegative, $(\F_t)$-progressively measurable processes with
$\E\left[\int_0^T \varphi_t^p\,{\rm d}t\right]<+\infty$ and
$\E\left[\left(\int_0^T f_t\,{\rm d}t\right)^p\right]<+\infty$, and
$\psi(\cdot,\cdot)\in {\bf S}[T,a(\cdot),b(\cdot)]$.\vspace{0.1cm}

\begin{Remark}\label{remark1}
  If $\psi (\cdot,\cdot)\in {\bf S}[T,a(\cdot),b(\cdot)]$ and $y_t\in \s^p$, then
  $$\E\left[\int_0^T \psi (t,|y_t|^p)\ {\rm d}t\right]\leq \int_0^T a(t)\ {\rm d}t+\int_0^T b(t)\
  {\rm d}t\cdot\E\left[\sup\limits_{t\in\T}|y_t|^p\right]<+\infty.$$
\end{Remark}

\begin{proposition}\label{proposition1}
  Let assumption (A) hold and let $(y_t,z_t)_{t\in\T}$ be a solution in $L^p$ to BSDE \eqref{BSDE1}. Denote $\bar \mu (t)=\int_t^T \mu^{p\over p-1}(s)\ {\rm d}s$ and  $\bar \nu(t)=\int_t^T \nu^2(s)\ {\rm d}s$. Then there exists a constant $m_p>0$ depending only on $p$ such that for each $t\in \T$,
  $$\begin{array}{lll}
      \Dis \E\left[\left(\int_t^T |z_s|^2\ {\rm d}s\right)^{p/2}\right]&\leq &\Dis m_p|\xi|^p+m_pC_t
         \left\{\E\left[\sup\limits_{s\in [t,T]}|y_s|^p\right]+\int_t^T
         \psi(s,\E[|y_s|^p])\ {\rm d}s\right.\\
      &  & \hspace{2.6cm}\Dis +\left.\E\left[\int_t^T \varphi_s^p\ {\rm d}s\right]+\E
           \left[\left(\int_t^T f_s\ {\rm d}s\right)^p\right]\right\},
    \end{array}
  $$
where $C_t:=1+\bar \mu^{p-1}(t)+\bar \mu^{2p-2} (t)+\bar \nu^{p/2} (t)+\bar \nu^{p} (t)$.
\end{proposition}

\begin{proof}
Applying It\^{o}'s formula to $|y_t|^2$ yields
$$
|y_t|^2+\int_t^T |z_s|^2\ {\rm d}s=|\xi|^2+2\int_t^T \langle
y_s,g(s,y_s,z_s)\rangle \ {\rm d}s-2\int_t^T\langle y_s,z_s{\rm
d}B_s\rangle. \vspace{-0.2cm}$$ It follows from assumption (A)
that for each $s\in [t,T]$,
$$
\begin{array}{lll}
2\langle y_s,g(s,y_s,z_s)\rangle
& \leq & 2\left(\sup\limits_{s\in [t,T]}|y_s|\right)\left(\mu
(s)\left[\psi^{1\over p}(s,|y_s|^p)+\varphi_s\right]+f_s\right)\\
&& \Dis +2 \left(\sup\limits_{s\in [t,T]}|y_s|^2\right)\cdot
\nu^2(s)+{|z_s|^2\over 2}.
\end{array}
$$
Thus, in view of the inequality $2ab\leq a^2+b^2$, we get that
$$
\begin{array}{lll}
\Dis {1\over 2}\int_t^T |z_s|^2\ {\rm d}s &\leq & \Dis
|\xi|^2+\left(2+2\bar\nu (t)\right)\cdot\left(\sup\limits_{s\in
[t,T]}|y_s|^2\right) +\left[\int_t^T f_s\ {\rm d}s\right]^2\\
&& \Dis \ \ \ \ \ +\left\{\int_t^T \mu
(s)\left[\psi^{1\over p}(s,|y_s|^p)+\varphi_s\right]\ {\rm
d}s\right\}^2+2\left|\int_t^T\langle y_s,z_s{\rm
d}B_s\rangle\right|.
\end{array}
$$
It follows from H\"{o}lder's inequality and the inequality $(a+b)^{p}\leq 2^p(a^{p}+b^{p})$ that
$$\hspace{-0.2cm}
\begin{array}{lll}
\Dis \left[\int_t^T \mu (s)\left[\psi^{1\over p}(s,|y_s|^p)+\varphi_s\right]\,{\rm d}s\right]^p &
\leq & \Dis
\left[\int_t^T \mu^{\frac{p}{p-1}}(s)\ {\rm d}s\right]^{p-1}\cdot
\int_t^T \left[\psi^{1\over p}(s,|y_s|^p)+\varphi_s\right]^p{\rm d}s\\
&\leq & \Dis  \bar\mu^{p-1}(t)\cdot 2^p\int_t^T
\left[\psi(s,|y_s|^p)+\varphi_s^p\right]\ {\rm d}s.
\end{array}
$$
Then there exists a constant $a_p>0$ depending only on $p$ such that
\begin{equation}\label{equation2}
\begin{array}{lll}
\hspace*{-0.2cm}\Dis \left[\int_t^T |z_s|^2\ {\rm d}s\right]^{p/2}&\leq & \Dis a_p|\xi|^p+c_t\left\{\sup\limits_{s\in
[t,T]}|y_s|^p+\int_t^T \psi(s,|y_s|^p)\ {\rm
d}s+\int_t^T \varphi_s^p\ {\rm d}s\right.\\
&& \Dis \hspace*{2.0cm}+\left.\left[\int_t^T
f_s\ {\rm d}s\right]^p+\left|\int_t^T\langle y_s,z_s{\rm
d}B_s\rangle\right|^{p/2}\right\},
\end{array}
\end{equation}
where $ c_t:=a_p\left(1+\bar \mu^{p-1} (t)+\bar \nu^{p/2}(t)\right)$.\vspace{0.2cm}

Furthermore, the Burkholder--Davis--Gundy (BDG) inequality implies that there exists a constant $d_p>0$ depending only on $p$ such that for each $t\in\T$,
$$
\begin{array}{lll}
\Dis c_t\E\left[\left|\int_t^T\langle y_s,z_s{\rm
d}B_s\rangle\right| ^{p/2} \right]&\leq & \Dis c_td_{p}\E\left[\sup\limits_{s\in
[t,T]}|y_s|^{p/2}\cdot\left( \int_t^T|z_s|^2\ {\rm
d}s\right)^{p/4}\right]\\
&\leq & \Dis {c_t^2d_{p}^2\over
2}\E\left[\sup\limits_{s\in [t,T]}|y_s|^p\right]+{1\over
2}\E\left[ \left(\int_t^T|z_s|^2\ {\rm d}s\right)^{p/2}\right].
\end{array}
$$
Returning to the estimate \eqref{equation2}, we get that for each $t\in \T$,
$$\begin{array}{lll}
\Dis \E\left[\left(\int_t^T |z_s|^2\ {\rm
d}s\right)^{p/2}\right]&\leq &\Dis 2a_p|\xi|^p+(2c_t+c_t^2d_p^2)
\left\{\E\left[\sup\limits_{s\in [t,T]}|y_s|^p\right]+\E\left[\int_t^T \varphi_s^p\ {\rm d}s\right]\right.\\
&& \Dis \hspace*{1.3cm}+\left.\E\left[\int_t^T
\psi(s,|y_s|^p)\ {\rm d}s\right]+\E\left[\left(\int_t^T f_s\ {\rm
d}s\right)^p\right]\right\}.
\end{array}
$$
Then it follows from the definition of the function $c_t$ that there exists a constant $b_p>0$ depending only on $p$ such that
$$2c_t+c_t^2d_p^2\leq b_p \left(1+\bar \mu^{p-1} (t)+\bar \mu^{2p-2}
(t)+\bar \nu^{p/2} (t)+\bar \nu^{p} (t)\right).$$
Thus, by taking $m_p=2a_p+b_p$, in view of the fact that $\psi(s,\cdot)$ is a concave function for each $s\in\T$, the conclusion of Proposition \ref{proposition1} follows from Fubini's theorem and Jensen's inequality, completing the proof.
\end{proof}

\begin{proposition}\label{proposition2}
Let assumption (A) hold and let $(y_t,z_t)_{t\in\T}$ be a solution in $L^p$ to BSDE \eqref{BSDE1}. Denote $\bar \mu (t)=\int_t^T \mu^{p\over p-1}(s)\ {\rm d}s$ and $\bar \nu(t)=\int_t^T \nu^2(s)\ {\rm d}s$. Then there exists a constant $k_p>0$ depending
only on $p$ with $K_t:=e^{k_{p}(\bar\mu (t)+\bar\nu(t))}$ such that for each $t\in \T$,
$$
\begin{array}{lll}
\Dis \E\left[\sup\limits_{s\in [t,T]}|y_s|^{p}\right] &\leq &\Dis
K_t\left\{k_p\E[|\xi|^p]+k_p \E\left[\left(\int_t^T f_s\ {\rm
d}s\right)^{p}\right]\right.\\
&& \Dis \hspace{1.0cm}+\left.{1\over 2}\E\left[\int_t^T
\varphi_s^p\ {\rm d}s\right]+{1\over 2}\int_t^T
\psi(s,\E[|y_s|^p])\ {\rm d}s \right\}.
\end{array}$$
\end{proposition}

\begin{proof}
Assumption (A) yields that $\langle y_s,g(s,y_s,z_s)\rangle\leq |y_s|\{\mu(s)[\psi^{{1\over p}}(s,|y_s|^p)+\varphi_s]+\nu(s)|z_s|+f_s\}$, from which and Lemma \ref{lemma2} we deduce that, with probability one, for each $t\in \T$,
$$\begin{array}{l}
\Dis |y_t|^p+c(p)\int_t^T |y_s|^{p-2}1_{|y_s|\neq 0}|z_s|^2\
{\rm d}s \leq \Dis|\xi|^p -p\int_t^T |y_s|^{p-2}1_{|y_s|\neq
0}\langle y_s,z_s{\rm d}B_s\rangle\\
\Dis\hspace*{3.7cm} +p\int_t^T |y_s|^{p-1}\left\{\mu(s)\left[\psi^{1\over p}(s,|y_s|^p)+
\varphi_s\right]+\nu(s)|z_s|+f_s\right\}\ {\rm d}s.
\end{array}$$
From Young's inequality ($a^rb^{1-r}\leq ra+(1-r)b$ for each $a\geq 0$, $b\geq 0$ and $0<r<1$) and the inequality $(a+b)^p\leq 2^p(a^p+b^p)$ it follows that
$$\begin{array}{lll}
&&\Dis p\int_t^T |y_s|^{p-1}\mu(s)\left(\psi^{1\over p}(s,|y_s|^p)+\varphi_s\right)\ {\rm d}s\\
&\leq & \Dis  (p-1)\delta^{\frac{1}{p-1}}\int_t^T
|y_s|^{p}\mu^{\frac{p}{p-1}}(s)\ {\rm d}s+{2^p\over \delta} \int_t^T
\left(\psi(s,|y_s|^p)+\varphi_s^p\right)\ {\rm d}s,
\end{array}$$
where $\delta>0$ will be chosen later. Thus, by assumption (A) and Remark \ref{remark1} we deduce first from the previous two inequalities that,
$\int_0^T |y_s|^{p-2}1_{|y_s|\neq 0}|z_s|^2\ {\rm d}s<+\infty,\ps$. Moreover, from the inequality $ab\leq(a^2+b^2)/2 $ we get that
$$
\Dis p\nu (s) |y_s|^{p-1}|z_s|\leq \Dis {p\nu^2(s)\over 1\wedge (p-1)}|y_s|^p +{c(p)\over
2}|y_s|^{p-2}1_{|y_s|\neq 0}|z_s|^2.
$$
Then for each $t\in \T$, we have
\begin{equation}\label{equation3}
\Dis |y_t|^p+{c(p)\over 2}\int_t^T |y_s|^{p-2}1_{|y_s|\neq
0}|z_s|^2\ {\rm d}s \leq  \Dis X_t-p\int_t^T
|y_s|^{p-2}1_{|y_s|\neq 0}\langle y_s,z_s{\rm d}B_s\rangle,
\end{equation}
where $$X_t:=|\xi|^p +d_{p,\delta} \int_t^T \left(\mu^{p\over p-1}(s)+\nu^2(s)\right)|y_s|^{p}\ {\rm d}s+
 {2^p\over \delta} \int_t^T (\psi(s,|y_s|^p)+\varphi_s^p)\ {\rm d}s+
 p\int_t^T |y_s|^{p-1}f_s\ {\rm d}s
$$
with $d_{p,\delta}=(p-1)\delta^{1/(p-1)} +{p/[1\wedge
(p-1)]}>0$.\vspace{0.2cm}

It follows from the BDG inequality that $\{M_t:=\int_0^t
|y_s|^{p-2}1_{|y_s|\neq 0}\langle y_s,z_s{\rm
d}B_s\rangle\}_{t\in\T}$ is a uniformly integrable martingale. In
fact, by Young's inequality we have
$$
\begin{array}{lll}
\Dis \E\left[\langle M, M \rangle^{1/2}_T\right]&\leq &
\Dis\E\left[\sup\limits_{s\in [0,T]}|y_s|^{p-1}\cdot\left(
\int_0^T|z_s|^2\ {\rm
d}s\right)^{1/2}\right]\\
&\leq &\Dis {(p-1)\over p}\E\left[\sup\limits_{s\in
[0,T]}|y_s|^p\right]+{1\over p}\E\left[\left( \int_0^T|z_s|^2\
{\rm d}s\right)^{p/2}\right]<+\infty.
\end{array}
$$
Returning to inequality \eqref{equation3} and taking the expectation, we get both
\begin{equation}\label{equation4}
{c(p)\over 2}\E\left[\int_t^T |y_s|^{p-2}1_{|y_s|\neq 0}|z_s|^2\
{\rm d}s\right] \leq \E[X_t]
\end{equation}
and
\begin{equation}\label{equation5}
\E\left[\sup\limits_{s\in [t,T]}|y_s|^{p}\right]\leq \E[X_t]+\bar
k_p\E\left[\left(\langle M, M \rangle_T-\langle M, M
\rangle_t\right) ^{1/2}\right],
\end{equation}
where the $\bar k_p>0$ only depends on $p$. The last step uses the BDG inequality.\vspace{0.2cm}

On the other hand, Young's inequality implies that
$$
\begin{array}{lll}
&&\Dis \bar k_p\E\left[\left(\langle M, M \rangle_T-\langle M, M
\rangle_t\right) ^{1/2}\right]\\
&\leq & \Dis \bar k_p\E\left[\sup\limits_{s\in
[t,T]}|y_s|^{p/2}\cdot\left( \int_t^T|y_s|^{p-2}1_{|y_s|\neq
0}|z_s|^2\ {\rm
d}s\right)^{1/2}\right]\\
&\leq &\Dis {1\over 2}\E\left[\sup\limits_{s\in
[t,T]}|y_s|^p\right]+{\bar k_p^2\over 2}\E\left[
\int_t^T|y_s|^{p-2}1_{|y_s|\neq 0}|z_s|^2\ {\rm d}s\right].
\end{array}
$$
We may now combine inequalities \eqref{equation4} and \eqref{equation5} to obtain the existence of a constant $k'_p>0$ such that $\E\left[\sup_{s\in [t,T]}|y_s|^{p}\right]\leq k'_p\E[X_t]$. Another application of Young's inequality yields the existence of a constant $k''_p>0$ depending only on $p$ such that
$$\begin{array}{lll}
\Dis pk'_p\E\left[\int_t^T |y_s|^{p-1}f_s\ {\rm d}s \right]&\leq
&\Dis pk'_p\E\left[\sup\limits_{s\in
[t,T]}|y_s|^{p-1}\int_t^T f_s\ {\rm d}s \right]\\
&\leq & \Dis  {1\over 2}\E\left[\sup\limits_{s\in
[t,T]}|y_s|^{p}\right]+{k''_p\over 2}\E\left[\left(\int_t^T f_s\
{\rm d}s\right)^{p}\right].
\end{array}$$
Thus, using the definition of $X_t$ we can deduce that
$$\begin{array}{lll}
\Dis \E\left[\sup\limits_{s\in [t,T]}|y_s|^{p}\right] &\leq & \Dis
2k'_p\E\left[|\xi|^p +d_{p,\delta} \int_t^T \left(\mu^{p\over
p-1}(s)+\nu^2(s)\right)|y_s|^{p}\ {\rm d}s\right]\\
 && \Dis +2k'_p\E\left[
 {2^p\over \delta}\int_t^T \left(\psi(s,|y_s|^p)+\varphi_s^p\right)
 \ {\rm d}s\right]+k''_p\E\left[\left(\int_t^T f_s\
{\rm d}s\right)^{p}\right].
\end{array}$$
By letting $\delta=2^{p+2}k'_p$ and $h_t=\E\left[\sup_{s\in [t,T]}|y_s|^{p}\right]$
in the previous inequality and using Fubini's theorem and Jensen's inequality, noticing that $\psi(s,\cdot)$ is a concave function for each $s\in \T$, we know that for each $t\in \T$,
$$\begin{array}{lll}
\Dis h_t &\leq & \Dis  2k'_p\E[|\xi|^p]
 +k''_p\E\left[\left(\int_t^T f_s\
{\rm d}s\right)^{p}\right]+{1\over 2}\E\left[\int_t^T \varphi_s^p\
{\rm d}s\right]\\
&& \Dis +{1\over 2}\int_t^T \psi\left(s,\E[|y_s|^p]\right)\ {\rm
d}s+2k'_pd_{p,\delta}\int_t^T \left(\mu^{p\over p-1}(s)+\nu^2(s)\right)h_s\
{\rm d}s.
\end{array}
$$
Finally, in view of assumption (A), the Backward Gronwall inequality (Lemma \ref{lemma1}) yields that for each $t\in \T$,
$$\begin{array}{lll}
\Dis h_t&\leq & \Dis e^{2k'_pd_{p,\delta}(\bar \mu(t)+\bar\nu
(t))}\left\{2k'_p\E[|\xi|^p]
 +k''_p\E\left[\left(\int_t^T f_s\
{\rm d}s\right)^{p}\right]\right.\\
&& \Dis \hspace{3.0cm}+\left.{1\over 2}\E\left[\int_t^T
\varphi_s^p\ {\rm d}s\right]+
 {1\over 2}\int_t^T \psi(s,\E[|y_s|^p])\ {\rm d}s\right\}.
 \end{array}
$$
The proof of Proposition \ref{proposition2} is thus completed.\vspace{-0.4cm}
\end{proof}

\section{Main Result and Its Proof}\label{sectionmainresult}\setcounter{equation}{0}

In this section, we will put forward and prove our main result. Let us first introduce the following assumptions, where we assume that $0<T\leq +\infty$:
\begin{enumerate}
\item[(H4)]
$\as,\ \RE y_1,y_2\in \R^k,z_1,z_2\in\R^{k\times d}$,
$$|g(\omega,t,y_1,z_1)-g(\omega,t,y_2,z_2)|\leq\alpha(t)\rho^{1\over p}(t,|y_1-y_2|^p)+\beta(t)|z_1-z_2|,$$
\end{enumerate}
where $\alpha(\cdot)$, $\beta(\cdot):\T\mapsto \R^+$ satisfy the condition $\int_0^T\left(\alpha^{p\over p-1}(t)+\beta^2(t)\right)\ {\rm d}t<+\infty$ and the function $\rho(\cdot,\cdot)$ belongs to ${\bf S}[T,a(\cdot),b(\cdot)]$;
\begin{enumerate}
  \item[(H5)] $\Dis \E\left[\left(\int_0^T |g(t,0,0)|\ {\rm d}t\right)^p\right]<+\infty$.
\end{enumerate}

\begin{Remark}\label{remark2}
It follows from Young's inequality, the inequality $(a+b)^p\leq 2^p(a^p+b^p)$, and assumption (H4) that
  $$\Dis\int_0^T\left[\alpha(t)\left(a^{1\over p}(t)+b^{1\over p}(t)\right)\right]{\rm d}t
  \leq \int_0^T \left[{p-1\over p}\alpha^{p\over p-1}(t)+{2^p\over p}(a(t)+b(t))\right]{\rm
  d}t<+\infty.$$
Furthermore, H\"{o}lder's inequality yields that $g(\cdot,0,0)\in {\rm M}^p(0,T;\R^{k})$ implies (H5) in the case where $T<+\infty$.
\end{Remark}

The following Theorem \ref{theoremmainresult} is the main result of this paper.

\begin{theorem}\label{theoremmainresult}
  Let $0<T\leq +\infty$ and $g$ satisfy (H4) and (H5). Then, for each $\xi\in\Lp$, the BSDE with parameters $(\xi,T,g)$ has a unique solution in $L^p$.
\end{theorem}

In order to prove Theorem \ref{theoremmainresult}, we need first to establish the following Proposition \ref{proposition3}, which is just Theorem 1.2 of \citet{Chen00} when $p=2$.

\begin{proposition}\label{proposition3}
  Let $0<T\leq +\infty$ and $g$ satisfy (H3) and (H5). Then, for each $\xi\in\Lp$, the BSDE with parameters $(\xi,T,g)$ has a unique solution in $L^p$.
\end{proposition}

\begin{proof}
Define $\hat u([t_1,t_2]):=\int_{t_1}^{t_2} u(s){\rm d}s$ and
$\hat v([t_1,t_2]):=\int_{t_1}^{t_2}v^2(s){\rm d}s$ for each $0\leq t_1\leq t_2\leq T$.
Assume that $(y_t,z_t)_{t\in\T}\in{\s}^p(0,T;\R^{k})\times {\rm M}^p(0,T;\R^{k\times d})$. It follows from assumption (H3) that
$|g(s,y_s,z_s)|\leq|g(s,0,0)|+u(s)|y_s|+v(s)|z_s|$
and then from the inequality $(a+b+c)^p\leq 3^p(a^p+b^p+c^p)$ and H\"{o}lder's inequality that for each $t\in \T$,
$$
\begin{array}{lll}
\Dis\E\left[\left(\int_t^T|g(s,y_s,z_s)|\ {\rm d}s\right)^p\right]
&\leq & \Dis {3^p}\E\left[\left(\int_t^T|g(s,0,0)|\ {\rm
d}s\right)^p\right]+{3^p}\hat u^p([t,T])\cdot \E\left[\sup\limits_{s\in \T}|y_s|^p\right]\\
&& \Dis \ +{3^p} \hat v^{p\over 2}([t,T])\cdot\E\left[\left(\int_t^T |z_s|^2\ {\rm d}s\right)^{p/2}\right]<+\infty,
\end{array}
$$
As a result, the process $\left\{\E\left[\left.\xi+\int_0^Tg(s,y_s,z_s) {\rm d}s\right|\F_t\right]\right\}_{0\leq t\leq T}$ is an $L^p$ martingale. It follows from the martingale representation theorem that there exists a unique process $Z_t\in {\rm M}^p(0,T;\R^{k\times d})$ such that
$$
\E\left[\left.\xi+\int_0^Tg(s,y_s,z_s)\ {\rm
d}s\right|\F_t\right]=\E\left[\xi+\int_0^Tg(s,y_s,z_s)\ {\rm
d}s\right]+\int_0^t Z_s\ {\rm d}B_s,\ 0\leq t\leq T.
$$
Let
$Y_t:=\E\left[\left.\xi+\int_t^Tg(s,y_s,z_s)\ {\rm d}s\right|\F_t\right],\ 0\leq t\leq T$.
Obviously, $Y_t\in{\s}^p(0,T;\R^{k})$. It is not difficult to verify that the $(Y_t,Z_t)_{t\in\T}$ is just the unique solution in $L^p$ to the following equation:
\begin{equation}\label{equation6}
\Dis Y_t=\xi+\int_t^T g(s,y_s,z_s)\ {\rm d}s-\int_t^T Z_s\ {\rm d}B_s,\ \  t\in\T.
\end{equation}
Thus, we have constructed a mapping from ${\s}^p(0,T;\R^{k})\times{\rm M}^p(0,T;\R^{k\times d})$
to itself. Denote this mapping by $\Phi : (y_{\cdot},z_\cdot)\To(Y_{\cdot},Z_\cdot)$.\vspace{0.2cm}

In the sequel, suppose that
$(y_t^i,z_t^i)_{t\in\T}\in{\s}^p(0,T;\R^{k})\times {\rm M}^p(0,T;\R^{k\times d})$, let $(Y_t^i,Z_t^i)_{t\in\T}$ be the mapping of $(y_t^i,z_t^i)_{t\in\T}$, $(i=1,2)$, that is
$\Phi(y_\cdot^i,z_\cdot^i)=(Y_t^i,Z_t^i),\ i=1,2$. We denote
$\hat Y_t:=Y_t^1-Y_t^2, \ \hat Z_t:=Z_t^1-Z_t^2,\ \hat y_t:=y_t^1-y_t^2, \hat z_t:=z_t^1-z_t^2,
\ \hat g_t:=g(t,y^1_t,z^1_t)-g(t,y^2_t,z^2_t)$.
It follows from \eqref{equation6} that
\begin{equation}\label{equation7}
\Dis \hat Y_t=\int_t^T \hat g_s\ {\rm d}s-\int_t^T \hat Z_s\ {\rm d}B_s,\ \ t\in\T.
\end{equation}
Assumption (H3) yields that $|\hat g_t|\leq u(t)|\hat y_t|+v(t)|\hat z_t|$, which means that the generator $\hat g_t$ of BSDE \eqref{equation7} satisfies assumption (A) with $\mu(t)=u^{p-1\over p}(t)$, $\nu(t)\equiv 0$, $\psi(t,u)\equiv 0$, $\varphi_t=u^{1\over p}(t)|\hat y_t|$ and $f_t=v(t)|\hat z_t|$ due to the fact that
\vspace{-0.1cm}
\begin{equation}\label{equation8}
\left\{
\begin{array}{l}
\Dis \E\left[\int_t^T \left(u^{1\over p}(s)|\hat y_s|\right)^p\ {\rm d}s\right]\leq \hat u([t,T])\cdot\E\left[\sup\limits_{s\in [t,T]}|\hat y_s|^p \right]<+\infty;\\
\Dis \E\left[\left(\int_t^T v(s)|\hat z_s|\ {\rm
d}s\right)^p\right]\leq \hat v^{p\over 2}([t,T])\cdot\E\left[\left(\int_t^T
|\hat z_s|^2\ {\rm d}s\right)^{p/2}\right]<+\infty
\end{array}
\right.
\end{equation}
is true for each $t\in \T$ by H\"{o}lder's inequality. Thus, applying Propositions \ref{proposition1}--\ref{proposition2} to BSDE \eqref{equation7} implies that, in view of \eqref{equation8}, there exists a constant $m'_p>0$ depending only on $p$ such that for
each $t\in \T$,
$$
\left\{\begin{array}{l}
\begin{array}{lll}
\Dis \E\left[\left(\int_t^T |\hat Z_s|^2\ {\rm d}s\right)^{p/2}\right]&\leq &\Dis m'_p C'([t,T])
\left\{\hat
v^{p/2}([t,T])\cdot\E\left[\left(\int_t^T |\hat z_s|^2\ {\rm
d}s\right)^{p/2}\right]\right.\\
&& \Dis \hspace*{1.5cm}+\left.\hat
u([t,T])\cdot\E\left[\sup\limits_{s\in [t,T]}|\hat y_s|^p
\right]+\E\left[\sup\limits_{s\in [t,T]}|\hat Y_s|^p
\right]\right\},
\end{array}\\
\begin{array}{l}
\Dis \E\left[\sup\limits_{s\in [t,T]}|\hat Y_s|^{p}\right] \leq \Dis
K'([t,T])\left\{{1\over 2}\hat u([t,T])\cdot \E\left[\sup\limits_{s\in
[t,T]}|\hat y_s|^p \right]\right.\\
\hspace*{5.0cm}\Dis \left.+m'_{p}\hat
v^{p/2}([t,T])\cdot\E\left[\left(\int_t^T |\hat z_s|^2\ {\rm
d}s\right)^{p/2}\right]\right\}.
\end{array}
\end{array}\right.
$$
where
$C'([t,T]):=1+\hat u^{p-1}([t,T])+\hat u^{2p-2}([t,T])$ and
$K'([t,T]):=e^{m'_{p}\hat u([t,T])}$. Then we have
$$
\begin{array}{lll}
\Dis \E\left[\left(\int_t^T |\hat Z_s|^2\ {\rm d}s\right)^{p/2}\right]&\leq &\Dis m'_p C'([t,T])
\left\{({1\over 2}K'([t,T])+1)\hat
u([t,T])\cdot\E\left[\sup\limits_{s\in [t,T]}|\hat y_s|^p
\right]\right.\\
&& \Dis +\left.(K'([t,T])m'_p+1)\hat
v^{p/2}([t,T])\cdot\E\left[\left(\int_t^T |\hat z_s|^2\ {\rm
d}s\right)^{p/2}\right]\right\}.
\end{array}
$$

Since $\hat u([0,T])<+\infty$ and $\hat v([0,T])<+\infty$ by assumption (H3), we can find a positive integer $N$ and $0=T_0<T_1<\cdots<T_{N-1}<T_N=T$ such that for each $i=0,\cdots,N-1$,\vspace{-0.1cm}
\begin{equation}\label{equation9}
\left\{
\begin{array}{l}
\Dis K'([T_i,T_{i+1}]){\hat u([T_i,T_{i+1}])\over 2}+m'_pC'([T_i,T_{i+1}])\left({1\over 2}K'([T_i,T_{i+1}])+1\right)\hat
u([T_i,T_{i+1}])\leq {1\over 2};\vspace{0.1cm}\\
\Dis K'([T_i,T_{i+1}])m'_{p}\hat v^{p\over
2}([T_i,T_{i+1}])+m'_p C'([T_i,T_{i+1}])(K'([T_i,T_{i+1}])m'_p+1)\hat v^{p\over
2}([T_i,T_{i+1}])\leq {1\over 2}.
\end{array}\right.\vspace{0.1cm}
\end{equation}

Based on the above arguments, we can deduce that
$$
\begin{array}{lll}
&&\Dis \E\left[\sup\limits_{s\in [T_{N-1},T]}|\hat Y_s|^{p}\right] +\E\left[\left(\int_{T_{N-1}}^{T}
|\hat Z_s|^2\ {\rm d}s\right)^{p/2} \right]\\
&\leq & \Dis {1\over 2}\left\{ \E\left[\sup\limits_{s\in
[T_{N-1},T]}|\hat y_s|^{p}\right]
+\E\left[\left(\int_{T_{N-1}}^{T} |\hat z_s|^2\ {\rm
d}s\right)^{p/2} \right]\right\},
\end{array}
$$
which means that $\Phi$ is a strict contraction from
${\s}^p(T_{N-1},T;\R^{k})\tim {\rm M}^p(T_{N-1},T;\R^{k\tim d})$ into itself. Then $\Phi$ has a unique fixed point in this space. It follows that there exists a unique
$(y_t,z_t)_{t\in [T_{N-1},T]}\in{\s}^p(T_{N-1},T;\R^{k})\times {\rm M}^p(T_{N-1},T;\R^{k\tim d})$
satisfying the BSDE with parameters $(\xi,T,g) $ on $[T_{N-1},T]$. That is to say, the BSDE has a unique solution in $L^p$ on $[T_{N-1},T]$. Finally, note that \eqref{equation9} holds true
for $i=N-2$. By replacing $T_{N-1}$, $T$ and $\xi$ by $T_{N-2}$, $T_{N-1}$ and $y_{T_{N-1}}$ respectively in the above proof except for the paragraph containing \eqref{equation9}, we can
obtain the existence and uniqueness of a solution in $L^p$ to the BSDE with parameters $(\xi,T,g)$ on $[T_{N-2},T_{N-1}]$. Furthermore, repeating the above procedure and making
use of \eqref{equation9}, we deduce the existence and uniqueness of a solution in $L^p$ to the BSDE with parameters $(\xi,T,g)$ on $[T_{N-3},T_{N-2}]$, $\cdots$, $[0,T_1]$. The proof
of Proposition \ref{proposition3} is then completed.
\end{proof}

Now, we are in a position to prove Theorem \ref{theoremmainresult}. Let $0<T\leq +\infty$, $\xi\in \Lp$ and $g$ satisfy (H4) and (H5). We can construct the Picard approximation sequence
of the BSDE with parameters $(\xi,T,g)$ as follows:
\begin{equation}\label{equation10}
y_t^0=0;\ \ \ \  \Dis y_t^n=\xi+\int_t^Tg(s,y_s^{n-1},z_s^n)\ {\rm
d}s-\int_t^Tz_s^n\ {\rm d}B_s,\ \  t\in\T.
\end{equation}
Indeed, for each $n\geq 1$, it follows from assumption (H4) that
$$
\begin{array}{lll}
|g(s,y_s^{n-1},0)|&\leq & \Dis |g(s,0,0)|+\alpha(s)\rho^{1\over p}(s,|y_s^{n-1}|^p)\\
&\leq & \Dis |g(s,0,0)|+\alpha(s)\left(a^{1\over p}(s)+b^{1\over p}(s)|y_s^{n-1}|\right),
\end{array}
$$
and then
$$
\begin{array}{lll}
\Dis \E\left[\left(\int_0^T|g(s,y_s^{n-1},0)|\ {\rm
d}s\right)^p\right] &\leq & \Dis
3^p\E\left[\left(\int_0^T|g(s,0,0)|\ {\rm
d}s\right)^p\right]+3^p\left(\int_0^T \alpha(s)a^{1\over p}(s)\ {\rm
d}s\right)^p\\
&& \Dis +3^p\left(\int_0^T \alpha(s)b^{1\over p}(s)\ {\rm
d}s\right)^p\E\left[\sup\limits_{s\in\T} |y_s^{n-1}|^p\right] .\end{array}
$$
Furthermore, by Remark \ref{remark2} and assumption (H4), the generator $g(s,y_s^{n-1},z)$ of BSDE \eqref{equation10} satisfies (H5) and (H3) with $u(t)=0$ and $v(t)=\beta(t)$.
It follows from Proposition \ref{proposition3} that the equation \eqref{equation10} has a unique solution $(y_t^n,z_t^n)_{t\in\T}$ in $L^p$ for each $n\geq 1$. With respect to the processes $(y_t^n,z_t^n)_{t\in\T}$, we have the following Lemma \ref{lemma3} and Lemma \ref{lemma4}.  For notational convenience, in the following for each $0\leq t_1\leq t_2\leq T$, we define
$$\hat \alpha([t_1,t_2]):=\int_{t_1}^{t_2}\alpha^{p\over p-1}(s)\ {\rm d}s\
{\rm and} \ \ \hat\beta([t_1,t_2]):=\int_{t_1}^{t_2} \beta^2(s)\ {\rm d}s.
$$

\begin{lemma}\label{lemma3}
  Under the hypotheses of Theorem \ref{theoremmainresult}, there exists a constant $\bar m_p>0$ depending only on $p$ such that for each $t\in \T, n,m\geq 1$,
  \begin{equation}\label{equation11}
  \left\{
  \begin{array}{l}
  \begin{array}{lll}
  \Dis \E\left[\left(\int_t^T |z_s^{n+m}-z_s^n|^2\ {\rm
  d}s\right)^{p/2}\right]&\leq&\Dis\bar m_p\bar C([t,T])
  \left\{\E\left[\sup\limits_{s\in
  [t,T]}|y_s^{n+m}-y_s^n|^p\right]\right. \\
  && \Dis \hspace{0.7cm}+\left. \int_t^T \rho\left(s,\E\left[
  |y_s^{n+m-1}-y_s^{n-1}|^p\right]\right)\ {\rm d}s\right\};
  \end{array}\\
  \Dis \E\left[\sup\limits_{s\in
  [t,T]}|y_s^{n+m}-y_s^n|^p\right] \leq \Dis {1\over 2}\bar
  K([t,T])\int_{t}^T\rho\left(s,\E\left[
  |y_s^{n+m-1}-y_s^{n-1}|^p\right]\right){\rm d}s,
  \end{array}\right.
  \end{equation}
  where
  $$
  \left\{\begin{array}{l}
  \Dis \bar C([t,T]):=1+\hat \alpha^{p-1}([t,T])+\hat
  \alpha^{2p-2} ([t,T]) +\hat \beta^{p/2} ([t,T])+\hat \beta^{p} ([t,T]);\vspace{0.1cm}\\
  \Dis \bar K([t,T]):=e^{\bar m_{p}(\hat\alpha([t,T])+\hat\beta([t,T]))}.
  \end{array}\right.
  $$
\end{lemma}

\begin{proof}
It follows from \eqref{equation10} that the process
$(y_t^{n+m}-y_t^n,z_t^{n+m}-z_t^n)_{t\in\T}$ is a solution of the following BSDE:
\begin{equation}\label{equation12}
\Dis y_t=\int_t^T f_{n,m}(s,z_s)\ {\rm d}s-\int_t^Tz_s\ {\rm d}B_s,\ \ t\in\T,
\end{equation}
where $f_{n,m}(s,z):=g(s,y_s^{n+m-1},z+z_s^{n})-g(s,y_s^{n-1}, z_s^{n})$. It follows from assumption (H3) that
$|f_{n,m}(s,z)|\leq \alpha(s)\rho^{1\over p}(s,|y_s^{n+m-1}-y_s^{n-1}|^p)+\beta(s)|z|$,
which means that assumption (A) is satisfied for the generator $f_{n,m}(t,z)$ of BSDE \eqref{equation12} with $\mu(t)=\alpha(t)$, $\nu(t)=\beta (t)$, $\psi(t,u)\equiv 0$, $f_t\equiv 0$ and $\varphi_t=\rho^{1\over p}(t,|y_t^{n+m-1}-y_t^{n-1}|^p)$ by Remark \ref{remark1}. Thus, in view of the fact that $\rho(s,\cdot)$ is a concave function for each $s\in \T$, the conclusion \eqref{equation11} follows from Proposition \ref{proposition1}, Proposition \ref{proposition2}, and then Fubini's theorem and Jensen's inequality. Lemma \ref{lemma3} is proved.
\end{proof}

\begin{lemma}\label{lemma4}
Under the hypotheses of Theorem \ref{theoremmainresult},  there exists a constant $\hat m_p>0$ depending only on $p$ such that for each $n\geq 1$ and each $t\in [0,T]$,
\begin{equation}\label{equation13}
\E\left[\sup\limits_{r\in [t,T]}|y_r^n|^p\right]\leq  \Dis
\hat C([t,T])+{1\over 2}e^{\hat m_p(\hat\alpha([t,T])+\hat\beta([t,T]))}\int_{t}^T
\rho\left(s,\E\left[|y_s^{n-1}|^p\right]\right) \ {\rm d}s,
\end{equation}
where
$$
\hat C([t,T]):=\hat m_pe^{\hat m_p(\hat\alpha([t,T])+\hat\beta([t,T]))}\left\{\E|\xi|^p+
\E\left[\left(\int_{t}^T|g(s,0,0)|{\rm d}s\right)^p\right]\right\}.
$$
\end{lemma}

\begin{proof}
It follows from the hypotheses of Theorem \ref{theoremmainresult} that
$$
\begin{array}{lll}
|g(s,y_s^{n-1},z)| &\leq & |g(s,y_s^{n-1},z)-g(s,0,0)|+|g(s,0,0)|\\
 &\leq & \alpha(s) \rho^{1\over p}(s,|y_s^{n-1}|^p)+\beta(s)|z|+|g(s,0,0)|.
\end{array}
$$
Then, assumption (A) is satisfied for the generator $g(s,y_s^{n-1},z)$ of BSDE \eqref{equation10} with $\mu(t)=\alpha(t)$, $\nu(t)=\beta (t)$,
$\psi(t,u)\equiv 0$, $f_t= |g(t,0,0)|$ and $\varphi_t=\rho^{1\over p}(t,|y_t^{n-1}|^p)$ by Remark \ref{remark1}. Thus, in view of the
fact that $\rho(t,\cdot)$ is a concave function for each $t\in\T$, \eqref{equation13} follows from Proposition \ref{proposition2} and then Fubini's Theorem and Jensen's inequality. Lemma
\ref{lemma4} is proved.
\end{proof}

In the sequel, since $\hat \alpha([0,T])<+\infty$ and $\hat \beta([0,T])<+\infty$ by assumption (H4), we can find a positive integer $\bar N$ and $0=\bar T_0<\bar T_1<\cdots<\bar T_{\bar N-1}<\bar T_{\bar N}=T$ such that for each $i=0,\cdots,\bar N-1$,
\begin{equation}\label{equation14}
\int_{\bar T_i}^{\bar T_{i+1}} b(s){\rm d}s\leq {1\over 2}\ \ {\rm and}\ \
\hat\alpha([\bar T_i, \bar T_{i+1}])+\hat\beta([\bar T_i, \bar T_{i+1}])\leq {\ln 2\over \hat m_p}\wedge{\ln
2\over \bar m_p},
\end{equation}
where $\bar m_p$ and $\hat m_p$ are respectively defined in Lemma \ref{lemma3} and Lemma \ref{lemma4}.\vspace{0.2cm}

With the help of Lemma \ref{lemma3} and Lemma \ref{lemma4}, we can prove Theorem \ref{theoremmainresult}.\vspace{0.2cm}

{\bf Proof of Theorem \ref{theoremmainresult}.}
Existence: \ Let us set
$M=2\hat C([0,T])+2\int_{0}^T a(s){\rm d}s\geq 0$. It follows from (H4) and \eqref{equation14} that for each $t\in [\bar T_{\bar N-1},T]$,\vspace{-0.2cm}
\begin{equation}\label{equation15}
\hat C([0,T])+\int_{t}^T \rho(s,M){\rm d}s\leq \hat C([0,T])+\int_{t}^T a(s){\rm d}s +M\int_{t}^T b(s){\rm d}s \leq {M\over 2}+{M\over 2}=M,\ \
\end{equation}
and from Lemma \ref{lemma4} and \eqref{equation14} that
\begin{equation}\label{equation16}
\E\left[\sup\limits_{r\in [t,T]}|y_r^n|^p\right]\leq  \Dis
\hat C([0,T])+\int_{t}^T \rho\left(s,\E\left[|y_s^{n-1}|^p\right]\right) \
{\rm d}s,\ \ t\in [\bar T_{\bar N-1},T].
\vspace{-0.2cm}\end{equation}
Since $\rho(s, \cdot)$ is a nondecreasing function for each $s\in\T$, by \eqref{equation16} and \eqref{equation15} we can deduce that for each $t\in[\bar T_{\bar N-1},T]$, $\E\left[\sup_{r\in [t,T]}|y_r^1|^p\right]\leq \hat C([0,T])\leq M$,
\vspace{-0.2cm}
\begin{align*}
  \E\left[\sup_{r\in [t,T]}|y_r^2|^p\right]&\leq\hat C([0,T])+\int_{t}^T\rho(s,\E[|y_s^1|^p])\,{\rm d}s\leq\hat C([0,T])+\int_{t}^T\rho(s,M)\,{\rm d}s \leq M,\\
  \E\left[\sup_{r\in [t,T]}|y_r^3|^p\right]&\leq\hat C([0,T])+\int_{t}^T\rho(s,\E[|y_s^2|^p])\,{\rm d}s\leq\hat C([0,T])+\int_{t}^T\rho(s,M)\,{\rm d}s \leq M.
\end{align*}
Thus, by induction we know that for each $n\geq 1$ and each $t\in [\bar T_{\bar N-1},T]$,\vspace{-0.1cm}
\begin{equation}\label{equation17}
\E\left[\sup_{r\in [t,T]}|y_r^n|^p\right]\leq M.
\end{equation}\vspace{-0.1cm}

Now, we define a sequence of functions $\{\varphi_n(t)\}_{n\geq 1}$ as follows:
\begin{equation}\label{equation18}
\Dis \varphi_0(t)=\int_t^T\rho(s,M)\ {\rm d}s;\ \ \
\varphi_{n+1}(t)=\int_t^T\rho(s,\varphi_n(s))\ {\rm d}s.
\vspace{-0.2cm}\end{equation}
For all $t\in [\bar T_{\bar N-1},T]$, it follows from \eqref{equation15} that $\varphi_0(t)=\int_{t}^T\rho(s,M)\ {\rm d}s\leq M$. Furthermore, by induction we can obtain
that for all $n\geq 1$, $\varphi_n(t)$ satisfies
$0\leq \varphi_{n+1}(t)\leq \varphi_n(t)\leq \cdots\leq\varphi_{1}(t)\leq \varphi_0(t)\leq M$.
Then, for each $t\in [\bar T_{\bar N-1},T]$, the limit of the sequence $\{\varphi_n(t)\}_{n\geq 1}$ must exist: we denote it by $\varphi(t)$. Letting $n\To \infty$ in \eqref{equation18}, in view of the fact that $\rho(s,\cdot)$ is a continuous function for each $s\in\T$, $\rho(s,\varphi_n(s))\leq \rho(s,M)$ for each $n\geq 1$, and $\int_t^T \rho(s,M)\ {\rm d}s\leq M$, we can deduce from Lebesgue's dominated convergence theorem that for each $t\in [\bar T_{\bar N-1},T]$, $\varphi(t)=\int_t^T\rho(s,\varphi(s))\ {\rm d}s$, whether $T<+\infty$ or $T=+\infty$. Then, by virtue of (H4) we know that $\varphi(t)=0,\ t\in [\bar T_{\bar N-1},T]$.\vspace{0.2cm}

In the sequel, for each $t\in [\bar T_{\bar N-1},T]$ and $n,m\geq 1$, it follows from Lemma \ref{lemma3}, \eqref{equation14} and \eqref{equation17} that, whether $T<+\infty$ or $T=+\infty$,
$$
\begin{array}{l}
\Dis \E\left[\sup\limits_{r\in [t,T]}|y_r^{1+m}-y_r^1|^p\right]
\leq \Dis \int_t^T \rho\left(s,\E\left[|y_s^m|^p\right]\right)\
{\rm d}s
\leq \Dis \int_t^T\rho(s,M)\ {\rm d}s =\varphi_0(t)\leq M,\\
\Dis \E\left[\sup\limits_{r\in
[t,T]}|y_r^{2+m}-y_r^2|^p\right]\leq \Dis
\int_t^T\rho\left(s,\E\left[|y_s^{1+m}-y_s^1|^p\right]\right)\
{\rm d}s\leq \Dis \int_t^T\rho(s,\varphi_0(s))\ {\rm
d}s=\varphi_1(t),\\
\Dis \E\left[\sup\limits_{r\in
[t,T]}|y_r^{3+m}-y_r^3|^p\right]\leq \Dis \int_t^T
\rho\left(s,\E\left[|y_s^{2+m}-y_s^2|^p\right]\right)\ {\rm d}s
\leq \Dis \int_t^T\rho(s,\varphi_1(s))\ {\rm d}s=\varphi_2(t).
\end{array}
$$\vspace{-0.2cm}

Thus, by induction we can derive that for each $m\geq 1$,
$$\E\left[\sup\limits_{\bar T_{\bar N-1}\leq r\leq T}|y_r^{n+m}-y_r^n|^p\right]\leq \varphi_{n-1}(\bar T_{\bar N-1})\rightarrow 0,\ \ n\rightarrow \infty.\vspace{-0.2cm}
$$
which means that $\{y_t^n\}_{n\geq 1}$ is a Cauchy sequence in $\s^p(\bar T_{\bar N-1},T;\R^k)$. Furthermore, since $\rho(s,\cdot)$ is continuous and $\rho(s,0)=0$ for each $s\in \T$, $\int_{\bar T_{\bar N-1}}^T \rho(s,M){\rm d}s\leq M$, and $\rho(s,\E\left[|y_s^{n+m-1}-y_s^{n-1}|^p\right])\leq
\rho(s,M)$ for each $s\in [\bar T_{\bar N-1},T]$, we also know from \eqref{equation11}
and Lebesgue's dominated convergence theorem that $\{z_t^n\}_{n\geq 1}$ is a Cauchy sequence in ${\rm M}^p(\bar T_{\bar N-1},T;\R^{k\times d})$. Define their limits by $(y_t)_{t\in [\bar T_{\bar N-1},T]}$ and $(z_t)_{t\in [\bar T_{\bar N-1},T]}$ respectively. Letting $n\To \infty$ in \eqref{equation10} implies that $(y_t,z_t)$ is a solution in $L^p$ to the BSDE with parameters $(\xi,T,g)$ on $[\bar T_{\bar N-1},T]$.\vspace{0.2cm}

Finally, note that \eqref{equation14} holds true for $i=\bar N-2$. By replacing $\bar T_{\bar N-1}$, $T$ and $\xi$ with $\bar T_{\bar N-2}$, $\bar T_{\bar N-1}$ and $y_{\bar T_{\bar N-1}}$ respectively in the above arguments beginning from the end of the
proof of Proposition \ref{proposition3} (except for the paragraph containing \eqref{equation14}), we can obtain the existence of a solution in $L^p$ to the BSDE with parameters $(\xi,T,g)$
on $[\bar T_{\bar N-2},\bar T_{\bar N-1}]$. Furthermore, repeating the above procedure and making use of \eqref{equation14}, we deduce the existence of a solution in $L^p$ to the BSDE with parameters $(\xi,T,g)$ on $[\bar T_{\bar N-3},\bar T_{\bar N-2}]$, $\cdots$,
$[0,\bar T_1]$. This proves the existence.\vspace{0.3cm}

Uniqueness: Let $(y_t^i,z_t^i)_{t\in\T}\ (i=1,2)$ be two solutions in $L^p$ of the BSDE with parameters $(\xi,T,g)$. Then, $(y_t^1-y_t^2,z_t^1-z_t^2)_{t\in\T}$ is a solution in $L^p$ to
the following BSDE:
\begin{equation}\label{equation19}
\Dis y_t=\int_t^T \hat g(s,y_s,z_s)\ {\rm d}s-\int_t^Tz_s\ {\rm d}B_s,\ \ t\in\T,
\end{equation}
where $\hat g(s,y,z):=g(s,y+y_s^{2},z+z_s^{2})-g(s,y_s^{2},z_s^{2})$. It follows from (H4) that $|\hat g(s,y,z)|\leq \alpha(s)\rho^{1\over p}(s,|y|^p)+\beta(s)|z|$, which means that assumption (A) is satisfied for the generator $\hat g(t,y,z)$ of BSDE \eqref{equation19} with $\mu(t)=\alpha(t)$, $\nu(t)=\beta(t)$, $\psi(t,u)=\rho(t,u)$, $\varphi_t\equiv 0$ and $f_t\equiv 0$. Then,
Proposition \ref{proposition1} and Proposition \ref{proposition2} yield that there exists a constant $\tilde m_p>0$ depending only on $p$ such that for $t\in [0,T]$,
\begin{equation}\label{equation20}
\left\{
\begin{array}{l}
\Dis\E\left[\left(\int_t^T |z_s^1-z_s^2|^2\ {\rm
d}s\right)^{p/2}\right]\leq  \tilde m_p\tilde C([t,T])
\left\{\E\left[\sup\limits_{s\in [t,T]}|y_s^1-y_s^2|^p\right]\right.\\
\Dis \hspace*{7.5cm}\left.+\int_t^T \rho(s,\E[|y_s^1-y_s^2|^p])\ {\rm d}s\right\};\\
\Dis \E\left[|y_t^{1}-y_t^2|^p\right]\leq  \Dis {1\over 2}e^{\tilde m_p(\hat\alpha([t,T])+\hat\beta([t,T]))}\int_{t}^T\rho\left(s,\E\left[
|y_s^{1}-y_s^{2}|^p\right]\right)\ {\rm d}s,
\end{array}\right.
\end{equation}
where $\tilde C([t,T]):=1+\hat \alpha^{p-1}([t,T])+\hat
\alpha^{2p-2} ([t,T]) +\hat\beta^{p/2} ([t,T])+\hat\beta^{p} ([t,T])$.\vspace{0.3cm}

Similar to \eqref{equation14}, we can find a positive integer $\tilde N$ and
$0=\tilde T_0<\tilde T_1<\cdots<\tilde T_{\tilde N-1}<\tilde T_{\tilde N}=T$
such that for each $i=0,\cdots,\tilde N-1$,
\begin{equation}\label{equation21}
\hat\alpha([\tilde T_i, \tilde T_{i+1}])+\hat\beta([\tilde T_i, \tilde T_{i+1}])\leq {\ln 2\over \tilde m_p}.
\end{equation}
Then, it follows from \eqref{equation20} and \eqref{equation21} that for each
$t\in [\tilde T_{\tilde N-1},T]$,
\begin{equation}
\Dis \E\left[|y_t^{1}-y_t^2|^p\right]\leq  \Dis
\int_{t}^T\rho\left(s,\E\left[ |y_s^{1}-y_s^{2}|^p\right]\right)\
{\rm d}s.
\end{equation}
From the ODE comparison theorem, we know that $\E[|y_t^{1}-y_t^{2}|^p]\leq r(t)$, where $r(t)$ is the maximum left shift solution of the following equation:
$$ u'(t)=-\rho(t,u);\ u(T)=0.$$
It follows from (H4) that $r(t)=0,\ t\in [\tilde T_{\tilde N-1},T]$. Hence, $\E[|y_t^{1}-y_t^{2}|^p]=0,\ t\in [\tilde T_{\tilde N-1},T]$, which means $y_t^1=y_t^2$ for each $t\in [\tilde T_{\tilde N-1},T]$. Furthermore, \eqref{equation20} implies that $z_t^1=z_t^2$ holds true almost surely for each $t\in [\tilde T_{\tilde N-1},T]$. Thus, we have obtained the uniqueness result on $[\tilde T_{\tilde N-1},T]$. Then, thanks to \eqref{equation21}, we can repeat the above for the proof of uniqueness by replacing $\tilde T_{\tilde N-1}$, $T$ and $\xi$ with $\tilde T_{\tilde N-2}$, $\tilde T_{\tilde N-1}$ and $y_{\tilde T_{\tilde N-1}}$ respectively to obtain the uniqueness result on $[\tilde T_{\tilde N-2},\tilde T_{\tilde N-1}]$ and then on the whole $\T$. The proof of Theorem \ref{theoremmainresult} is complete.\vspace{-0.2cm} \hfill $\Box$

\section{Examples, Corollaries and Remarks}\label{sectionexamplecorollaryremark}\setcounter{equation}{0}
In this section, we will introduce some examples, corollaries and remarks to show that Theorem \ref{theoremmainresult} of this paper is a generalization of the main results in \citet{Par90}, \citet{Mao95}, \citet{Chen97}, \citet{Chen00}, \citet{Wang03} and \citet{Wang09}. Firstly, by Remark \ref{remark2} and Theorem \ref{theoremmainresult}, the following corollary is immediate, which generalizes the main results in \citet{Mao95}, \citet{Wang03} and \citet{Wang09}.

\begin{corollary}\label{corollary2}
Let $0<T<+\infty$ and $g$ satisfy (H4) and $g(\cdot,0,0)\in {\rm M}^p(0,T;\R^{k})$. Then, for each $\xi\in\Lp$, the BSDE with parameters $(\xi,T,g)$ has a unique solution in $L^p$.
\end{corollary}

\begin{Remark}\label{remark3}
Theorem \ref{theoremmainresult} in \citet{Wang09} proved that if
$0<T<+\infty $, $g$ satisfies assumption (H2), and
$g(\cdot,0,0)\in {\rm M}^2(0,T;\R^{k})$, then the BSDE with
parameters $(\xi,T,g)$ has a unique solution in $L^2$. This result
can be regarded as an immediate consequence of Corollary \ref{corollary2}. Indeed, it follows that if $g$ satisfies (H2), then $g$ must satisfy (H4) with $p=2$, $\alpha(t)\equiv 1$, $\beta(t)\equiv \sqrt{c}$ and $\rho(t,u)=\kappa(t,u)$.
\end{Remark}

Furthermore, let us introduce the following assumption, where we assume that $0<T\leq +\infty$:
\begin{enumerate}
\item[(H6)]
$\as,\ \RE y_1,y_2\in \R^k,z_1,z_2\in\R^{k\times d}$,
$$|g(\omega,t,y_1,z_1)-g(\omega,t,y_2,z_2)|\leq b(t)\bar\kappa^{1\over p}(|y_1-y_2|^p)+c(t)|z_1-z_2|,$$
\end{enumerate}
where $b(\cdot),c(\cdot):\T\mapsto \R^+$ satisfy $\int_0^T[b(t)+c^2(t)] \ {\rm d}t<+\infty$ and $\bar\kappa(\cdot)$ is a concave and nondecreasing function from $\R^+$ to $\R^+$ such that $\bar\kappa(0)=0$, $\bar\kappa(u)>0$ for $u>0$, and
$\int_{0^+}\bar\kappa^{-1}(u)\ {\rm d}u=+\infty$.

\begin{Remark}\label{remark4}
In next section, we will show that the concavity condition of $\bar\kappa(\cdot)$ in (H6) can be weakened to the continuity condition and that the bigger the $p$, the stronger the (H6).
\end{Remark}

The assumptions of $\bar \kappa(\cdot)$ in (H6) yield that there exists a constant $A>0$ such that for each $u\geq 0$, $\bar\kappa(u)\leq Au+A$. Then, from Theorem \ref{theoremmainresult} and Bihari's inequality, letting $\alpha(t)=b^{p-1\over p}(t)$, $\beta(t)=c(t)$ and $\rho(t,u)=b(t)\bar\kappa(u)\in {\bf S}[T,Ab(t),Ab(t)]$ in (H4), we can obtain the following corollary.

\begin{corollary}\label{corollary3}
Let $0<T\leq +\infty$ and $g$ satisfy (H5) and (H6). Then, for each $\xi\in\Lp$, the BSDE with parameters $(\xi,T,g)$ has a unique solution in $L^p$.
\end{corollary}

\begin{Remark}\label{remark5}
Theorem 2.1 in \citet{Mao95} proved that if $0<T<+\infty $,
$g(\cdot,0,0)\in {\rm M}^2(0,T;\R^{k})$ and $g$ satisfies assumption (H1), then the BSDE with parameters $(\xi,T,g)$ has a unique solution in $L^2$. This result can be regarded as an
immediate consequence of Corollary \ref{corollary3}. Indeed, it follows from Remark \ref{remark2} that under the above assumptions, the generator $g$ must satisfy (H5) and (H6)
with $p=2$, $b(t)\equiv 1$, $c(t)\equiv \sqrt{c}$ and $\bar\kappa(u)=\kappa(u)$.
\end{Remark}

\begin{example}\label{example1}
Let $0<T<+\infty$, and let
$$g(t,y,z)={1\over \sqrt{t}} h(|y|)+{1\over
\sqrt[4]{t}}|z|+|B_t|,$$
where
$h(x):=x|\ln x|^{1/p}\cdot 1_{0<x\leq\delta}+(h'(\delta-)(x-\delta)+h(\delta))\cdot 1_{ x>\delta}$
with $\delta>0$ small enough. It is not difficult to verify that $g$ satisfies assumptions
(H5) and (H6) with $b(t)={1/\sqrt t},c(t)={1/\sqrt[4]{t}}$ and $\bar\kappa(u)=h^p(u^{1/p})$. Thus, Corollary \ref{corollary3} yields that for each $\xi\in\Lp$, the BSDE with parameters $(\xi,T,g)$ has a unique solution in $L^p$.
\end{example}

\begin{Remark}\label{remark6}
Proposition \ref{proposition3} in the beginning of Section \ref{sectionmainresult} can also be regarded as an immediate consequence of Corollary \ref{corollary3}. Indeed, if $g$
satisfies (H3), then $g$ satisfies (H6) with $b(t)=u(t)$, $c(t)=v(t)$ and $\bar\kappa(u)=u$ for $u\geq 0$.
\end{Remark}

In the following, we introduce an example where $T$ can be $+\infty.$

\begin{example}\label{example2}
Let $0<T\leq +\infty$, and let
$$g(t,y,z)={1\over (1+t)^2}\sigma(|y|)+{1\over 1+t}|z|
+{1\over (1+t)^2},$$
where $\sigma(x):=x(|\ln x|\ln|\ln x|)^{1/p}\cdot 1_{0<x\leq
\delta}+(\sigma'(\delta-)(x-\delta)+\sigma(\delta))\cdot 1_{ x>\delta}$ with $\delta>0$ small enough. It is not difficult to verify that $g$ satisfies assumptions (H5) and (H6) with $b(t)=1/(1+t)^{2},c(t)=1/(1+t)$ and $\bar\kappa(u)=\sigma^p(u^{1/p})$. Thus, it follows from Corollary \ref{corollary3} that for each $\xi\in\Lp$, the BSDE with
parameters $(\xi,T,g)$ has a unique solution in $L^p$.
\end{example}

Finally, let us make Remark \ref{remark7}, which illustrates an important difference between the infinite and finite $T$ cases.

\begin{Remark}\label{remark7}
It is clear that in the case where $T<+\infty$, the $(y_t)_{t\in \T}$ among the solution $(y_t,z_t)_{t\in\T}$ of BSDEs discussed in this paper belongs also to ${\rm M}^p(0,T;\R^{k})$.
However, in the case where $T=+\infty$, this conclusion does not hold true. For a simple example, letting $\xi\equiv 1$, $T=+\infty$ and $g\equiv 0$, from Theorem \ref{theoremmainresult} or Corollary \ref{corollary3} we know that the BSDE with parameters
$(\xi,T,g)$ has a unique solution in $L^p$. Obviously, this solution is just $(1,0)_{0\leq t\leq +\infty}$. The process $(1)_{0\leq t\leq+\infty}$ belongs to $\s^p(0,T;\R^{k})$, but it does not belong to ${\rm M}^p(0,T;\R^{k})$.
\end{Remark}\vspace{-0.3cm}

\section{Further Discussion}\label{sectionfurtherdiscussion}\setcounter{equation}{0}

In this section, some further discussions with respect to our main result will be given. First, let us examine Remark \ref{remark4}. We need to show that the concavity condition of $\bar\kappa(\cdot)$ in (H6) can be weakened to the continuity condtion and
that the bigger the $p$, the stronger the (H6). To be precise, we need to prove that if $g$ satisfies the following assumption (H7') with $q\geq p$, then $g$ must satisfy the following
assumption (H7).\vspace{0.2cm}

{(H7)} There exists a deterministic function $b(t):\T\To \R^+$ with $\int_0^T b(t)\ {\rm d}t<+\infty$ and a nondecreasing and concave function $\kappa(\cdot):\R^+\mapsto \R^+$ with $\kappa(0)=0$, $\kappa(u)>0$ for $u>0$, and $\int_{0^+}\kappa^{-1}(u)\ {\rm d}u=+\infty$ such that $\as,$
$$\RE y_1,y_2\in \R^k,z\in\R^{k\times d},\ \
|g(\omega,t,y_1,z)-g(\omega,t,y_2,z)|\leq b(t)\kappa^{1\over p}(|y_1-y_2|^p).$$

{(H7')} There exists a deterministic function $\bar b(t):\T\To \R^+$ with $\int_0^T \bar b(t)\ {\rm d}t<+\infty$ and a nondecreasing and continuous function $\bar\kappa(\cdot):\R^+\mapsto \R^+$ with $\bar\kappa(0)=0$, $\bar\kappa(u)>0$ for $u>0$, and
$\int_{0^+}\bar\kappa^{-1}(u)\ {\rm d}u=+\infty$ such that $\as,$
$$\RE y_1,y_2\in \R^k,z\in\R^{k\times d},\ \
|g(\omega,t,y_1,z)-g(\omega,t,y_2,z)|\leq \bar
b(t)\bar\kappa^{1\over q}(|y_1-y_2|^q).$$

In order to show this fact, we need the following technical Lemma proved in the Appendix.

\begin{lemma}\label{lemma5}
Let $\rho(\cdot)$ be a nondecreasing and concave function on $\R^+$ with $\rho(0)=0$. Then we have
\begin{equation}\label{equation23}
  \RE\ r>1,\ \ \rho^r(x^{1/r})\
  {\rm is\ also\  a\ nondecreasing\ and \ concave\ function\ on}\ \R^+.
\end{equation}
Moreover, if $\rho(u)>0$ for $u>0$, and $\int_{0^+}\rho^{-1}(u)\ {\rm d}u=+\infty$, then
\begin{equation}\label{equation24}
  \RE\ r<1,\ \ \ \int_{0^+} {{\rm d}u\over \rho^r(u^{1/r})}=+\infty.
\end{equation}
\end{lemma}

Now, we can show that (H7')$\Longrightarrow$(H7). Let us assume that (H7') holds true for $g$. Then we have, $\as,$
$$
\RE y_1,y_2\in \R^k,z\in\R^{k\times d},\ \
|g(\omega,t,y_1,z)-g(\omega,t,y_2,z)|\leq \bar
b(t)\rho_1(|y_1-y_2|),
$$
where $\rho_1(u):=\bar \kappa^{1\over q}(u^q)$. Obviously, $\rho_1(u)$ is a continuous and nondecreasing function on $\R_+$ with $\rho_1(0)=0$ and $\rho_1(u)>0$ for $u>0$, but it is
not necessary to be concave. However, it follows from the classical theory of uniformly continuous functions that if $g$ satisfies the above inequality, then there exists a concave and nondecreasing function $\rho_2(\cdot)$ such that $\rho_2(0)=0$, $\rho_2(u)\leq 2\rho_1(u)$
for $u\geq 0$, and $\as,$
$$
\RE y_1,y_2\in \R^k,z\in\R^{k\times d},\ \
|g(\omega,t,y_1,z)-g(\omega,t,y_2,z)|\leq \bar
b(t)\rho_2(|y_1-y_2|).
$$
Thus, $\as,$
$$
\RE y_1,y_2\in \R^k,z\in\R^{k\times d},\ \
|g(\omega,t,y_1,z)-g(\omega,t,y_2,z)|\leq \bar b(t)\kappa^{1\over p}(|y_1-y_2|^p),
$$
where $\kappa(u):=\rho_2^p(u^{1\over p})+u$. It is clear that $\kappa(0)=0$ and $\kappa(u)>0$ for $u>0$. Moreover, it follows from \eqref{equation23} in Lemma \ref{lemma5} that $\kappa(\cdot)$ is also a nondecreasing and concave function due to the fact that $p>1$ and $\rho_2(\cdot)$ is a nondecreasing and concave function. Thus, to prove that (H7) holds, it suffices to
show that $\int_{0^+}\kappa^{-1}(u)\ {\rm d}u=+\infty$. Indeed, if $\rho_2(1)=0$, then since $\rho_2(u)=0$ for each $u\in[0,1]$, we have $\int_{0^+}\kappa^{-1}(u)\ {\rm d}u=\int_{0^+}u^{-1}\ {\rm d}u=+\infty$. On the other hand, if $\rho_2(1)>0$, since $\rho_2(\cdot)$ is a concave function with $\rho_2(0)=0$, we know that
\begin{equation}\label{equation25}
\rho_2(u)=\rho_2(u\cdot 1+(1-u)\cdot 0)\geq
u\rho_2(1)+(1-u) \rho_2(0)=u\rho_2(1),\ \ u\in [0,1],
\end{equation}
and then $\rho_2^{p}({u^{1\over p}})\geq \left(u^{1\over p} \rho_2(1) \right)^{p}=\rho_2^{p}(1)u$. Thus, we have
$$\RE\ u\geq 0,\ \ \kappa(u)=\rho_2^p(u^{1\over p})+u
\leq K\rho_2^p(u^{1\over p})\leq K2^p\rho_1^p(u^{1\over p})
=K2^p\bar \kappa^{p\over q}({u^{q\over p}}),$$
where $K=1+1/\rho_2^{p}(1)$. Consequently, if $q=p$, then
$$\int_{0^+}{{\rm d}u\over \kappa(u)}\geq {1\over K 2^p} \int_{0^+}{{\rm d}u\over\bar
\kappa(u)} =+\infty.$$
Thus, we have proved that (H7') with $q=p$ implies (H7). As a result, we can now assume that the $\bar \kappa (\cdot)$ in (H7') is a concave function. Then, if $q>p$, from \eqref{equation24} of Lemma \ref{lemma5} with $\rho(\cdot)=\bar \kappa(\cdot)$ and $r=p/q< 1$ we have
$$\int_{0^+}{{\rm d}u\over \kappa(u)}\geq {1\over
K 2^p} \int_{0^+}{{\rm d}u\over
 \bar \kappa^{p\over q}({u^{q\over p}})} =+\infty.$$
Hence (H7')$\Longrightarrow$(H7), i.e., the concavity condition of $\bar\kappa(\cdot)$ in (H6) can be weakened to the continuity condtion and the bigger the $p$, the stronger the (H6).\vspace{0.2cm}

Furthermore, let us introduce the following assumption (H6*), where we also assume that $0<T\leq +\infty$:

\begin{enumerate}
  \item[(H6*)] $\as,\ \RE y_1,y_2\in \R^k,z_1,z_2\in\R^{k\times d},$
       $$|g(\omega,t,y_1,z_1)-g(\omega,t,y_2,z_2)|\leq b(t)\kappa(|y_1-y_2|)+c(t)|z_1-z_2|,$$
\end{enumerate}
where $b(\cdot),c(\cdot):\T\mapsto \R^+$ satisfy $\int_0^T[b(t)+c^2(t)]\ {\rm d}t<+\infty$ and $\kappa(\cdot)$ is a continuous and nondecreasing function from $\T$ to $\R^+$ such that
$\kappa(0)=0$, $\kappa(u)>0$ for $u>0$, and $\int_{0^+}\kappa^{-p}(u)u^{p-1}\ {\rm d}u=+\infty$.\vspace{0.3cm}

In the following, we will show (H6*)$\Longrightarrow$(H6). In fact, if $g$ satisfies (H6*), then we have, $\as,\ \RE y_1,y_2\in\R^k,z_1,z_2\in\R^{k\times d}$,
$$|g(\omega,t,y_1,z_1)-g(\omega,t,y_2,z_2)|\leq
b(t)\bar\kappa^{1\over p}(|y_1-y_2|^p)+c(t)|z_1-z_2|,$$
where $\bar\kappa(u)=\kappa^p(u^{1/p})$. And, we have also
\begin{equation}\label{equation26}
\int_{0+}{{\rm d}u\over \bar\kappa(u)}=\int_{0+}{{\rm d}u \over
\kappa^p(u^{1\over p})}=\int_{0^+} {pu^{p-1}\over \kappa^p(u)}{\rm
d}u=+\infty.
\end{equation}
Thus, it follows from Remark \ref{remark4} that (H6) is true. Therefore, from Corollary \ref{corollary3} the following corollary is immediate. It follows from H\"{o}lder's inequality
that it generalizes the corresponding result in \citet{Cons01} where $p=2$, (H5) is replaced with $g(\cdot,0,0)\in {\rm M}^2(0,T;\R^{k})$, and $b(t)\equiv 1$, $c(t)\equiv c$ in (H6*).

\begin{corollary}\label{corollary4}
Let $0<T\leq +\infty$ and $g$ satisfy (H5) and (H6*). Then, for each $\xi\in\Lp$, the BSDE with parameters $(\xi,T,g)$ has a unique solution in $L^p$.
\end{corollary}

\begin{Remark}\label{remark8}
According to the classical theory of uniformly continuous functions, we can assume that the $\kappa(\cdot)$ in (H6*) is a concave function. Thus, applying \eqref{equation24} of Lemma
\ref{lemma5} yields, by letting $\rho(u)=\kappa^q(u^{1/q})$ and $r=p/q$ with $q>p$, that if $\int_{0+}\kappa^{-q}(u^{1/q})\ {\rm d}u=+\infty$, then $\int_{0+}\kappa^{-p}(u^{1/p})\ {\rm d}u=+\infty$. As a result, noticing \eqref{equation26} we know that the bigger the $p$, the stronger the (H6*).
\end{Remark}

Finally, it should be noted that the conclusions of Example \ref{example1} and Example \ref{example2} can also be obtained by virtue of Corollary \ref{corollary4}.\vspace{-0.3cm}

\section{Appendix}\label{sectionappendix}\setcounter{equation}{0}
{\bf Proof of Lemma \ref{lemma5}.} Assume first that $r>1$ and define $f(x)=\rho^r(x^{1/r})$. By means of approximation
procedures in \citet{Cons01} we know that in order to prove \eqref{equation23} it will be enough to show that $\RE\ x\geq 0,\ f''(x)\leq 0$ holds true for each function $\rho(\cdot)\in C^2(\R^+,\R^+)$ with $\rho(0)= 0$, $\rho'(x)\geq 0$ and $\rho''(x)\leq 0$ for $x\geq 0$. Indeed, we have
$f'(x)=\rho^{r-1}(x^{1\over r})\cdot x^{{1\over r}-1}\cdot \rho'(x^{1\over r})$, and then
$$f''(x)={(r-1)t\rho^{r-2}(t)\rho'(t)\over rx^2}
\left[\rho'(t)t-\rho(t)\right]+{t^2\rho^{r-1}(t)\rho''(t)\over
px^2}\vspace{-0.2cm}$$
with $t=x^{1/r}$. Considering that $\rho(t)\geq 0$, $\rho'(t)\geq 0$ and $\rho''(t)\leq 0$, it suffices to prove that $\rho'(t)t-\rho(t)\leq 0$. Note that Taylor's expansion yields that $0=\rho(0)=\rho(t)-t\rho'(t)+t^2{\rho''(\xi_t)/2}$ for $t>0$ and some $\xi_t\in (0,t)$. Since $\rho''(\xi_t)\leq 0$, the preceding relation proves $\rho'(t)t-\rho(t)\leq 0$. Then \eqref{equation23} is proved.\vspace{0.2cm}

Next we prove \eqref{equation24}. Let $r<1$, $\rho(u)>0$ for $u>0$, and $\int_{0^+}\rho^{-1}(u)\ {\rm d}u=+\infty$. Similar to \eqref{equation25} we know that $\RE\ u\in [0,1],\ \rho(u)\geq u\rho(1)$. Consequently, we have
$$\int_{0^+}{u^{\frac{1-r}{r}}\ {\rm d}u\over r\rho(u^{1\over r})}=\int_{0^+}{{\rm d}u\over
\rho(u)}=+\infty$$
and
$$\liminf_{u\To 0^+}{\frac{1}{\rho^{r}(u^{1\over r})} \over \frac{u^{\frac{1-r}{r}}}{\rho(u^{1\over r})}}
=\liminf_{u\To 0^+} \left(\rho(u^{1\over r})\over u^{1\over r}\right)^{1-r}\geq[\rho(1)]^{1-r}>0,$$
from which \eqref{equation24} follows immediately. The proof of Lemma \ref{lemma5} is completed. \vspace{-0.3cm} \hfill $\Box$

\section*{Acknowledgement(s)}
The authors are very grateful to the anonymous referees for their valuable comments and correcting some errors, and especially for their improving greatly the written language. The research has been supported by the National Natural Science Foundation of China (N0.
10971220), the FANEDD (N0.200919), the National Basic Research
Program of China (No. 2007CB814901), Qing Lan Project, and the Fundamental Research Funds for the Central Universities (No. 2010LKSX04).\vspace{-0.3cm}



\end{document}